\documentclass[preprint,review,10pt]{amsart}
\usepackage{amsmath,amssymb,amsfonts,xspace}
\newtheorem{theorem}{Theorem}[section]

\theoremstyle{definition}
\newtheorem{definition}[theorem]{Definition}
\newtheorem{example}[theorem]{Example}

\newtheorem{corollary}[theorem]{Corollary}
\newtheorem{lem}[theorem]{Lemma}

\theoremstyle{remark}

\numberwithin{equation}{section}

\begin{document}

\title{Remarks on $J$-hyperideals and their expansion
  }

\author{M. Anbarloei}
\address{Department of Mathematics, Faculty of Sciences,
Imam Khomeini International University, Qazvin, Iran.
}

\email{m.anbarloei@sci.ikiu.ac.ir }





\begin{abstract}
The aim of this research work is to define and characterize a new class of hyperideals in a Krasner $(m,n)$-hyperring that we call n-ary $J$-hyperideals. Also, we study the concept of n-ary $\delta$-$J$-hyperideals as an expansion of n-ary $J$-hyperideals. Finally, we extend the notion of n-ary $\delta$-$J$-hyperideals to $(k,n)$-absorbing $\delta$-$J$-hyperideals.
\end{abstract}
\maketitle
\section{Introduction}
Hyperstructures represent a natural extension of classical algebraic structures and they were
introduced by the French mathematician F. Marty.   In 1934,  Marty \cite{s1} defined  the concept of a hypergroup as a generalization of groups during the $8^{th}$ Congress of the Scandinavian Mathematicians. Many papers and books concerning hyperstructure theory have
appeared in literature  (see \cite {s2, s3, davvaz1, davvaz2, s4, jian}). The simplest algebraic hyperstructures which possess the properties of closure and associativity are called  semihypergroups. 
$n$-ary semigroups and $n$-ary groups are algebras with one $n$-ary operation which is associative and invertible in a generalized sense. The notion of investigations of $n$-ary algebras goes back to Kasner’s lecture \cite{s5} at a  scientific meeting in 1904.   In 1928, Dorente wrote the first paper concerning the theory of $n$-ary groups \cite{s6}. Later on, Crombez and
Timm \cite{s7, s8} defined  the notion of the $(m, n)$-rings and their quotient structures. Mirvakili and Davvaz [20] defined $(m,n)$-hyperrings and obtained several results in this respect. In \cite{s9}, they introduced a generalization of the notion of a hypergroup in the sense of Marty and a generalization of an $n$-ary group,   which is called $n$-ary hypergroup.
 The $n$-ary structures  has been studied in  \cite{l1, l2, l3, ma, rev1}. Mirvakili and Davvaz \cite{cons} defined $(m,n)$-hyperrings and obtained several results in this respect.

One important class of hyperrings was introduced by Krasner, where the addition is a hyperoperation, while the multiplication is an ordinary binary operation, which is called Krasner hyperring.  In \cite{d1},  a generalization of the Krasner hyperrings, which is a subclass of $(m,n)$-hyperrings, was defined by Mirvakili and Davvaz. It is called Krasner $(m,n)$-hyperring. Ameri and Norouzi in \cite{sorc1} introduced some important
hyperideals such as Jacobson radical, n-ary prime and primary hyperideals, nilradical, and n-ary multiplicative subsets of Krasner $(m, n)$-hyperrings. Afterward, the notions of $(k,n)$-absorbing hyperideals and $(k,n)$-absorbing primary hyperideals were studied by Hila et. al. \cite{rev2}.

Norouzi et. al.  proposed and analysed a new defnition for normal hyperideals in Krasner $(m,n)$-hyperrings, with respect to that one given in \cite{d1} and they showed that these hyperideals correspond to strongly regular relations \cite{nour}. In \cite{rev1}, Ostadhadi-Dehkordi  and Davvaz
deﬁned the fundamental relation $\eta *$ on $R$ as the smallest equivalence relation on $R$ such that the quotient
$[R: \eta *]$ is an $(m, n)$-ring. Asadi and Ameri  introduced and studied direct limit of a direct system in the category of Krasner $(m,n)$-hyperrigs \cite{asadi}.

Dongsheng  defined the notion of $\delta$-primary ideals in  a commutative ring  where  $\delta$ is a function that assigns to each ideal $I$  an ideal $\delta(I)$ of the same ring \cite{bmb2}.  Moreover, in \cite{bmb3} he and his colleague  investigated 2-absorbing $\delta$-primary ideals  which  unify 2-absorbing ideals and 2-absorbing primary
ideals. Ozel Ay et al.  generalized the notion of $\delta$-primary  on Krasner hyperrings \cite{bmb4}. The concept of $\delta$-primary  hyperideals in Krasner $(m,n)$-hyperrings, which unifies the prime and primary hyperideals under one frame, was defined in \cite{mah3}. The notion of $J$-ideals  as a generalization of n-ideals  in
ordinary rings was studied by Khashan and Bani-ata in \cite{ata}.

Now in this paper, first we  define    the notion of n-ary $J$-hyperideals in a Krasner $(m,n)$-hyperring which is a generalization of $J$-ideals.   We give several characterizations of n-ary $J$-hyperideals. 
Afterward, we study the concept of n-ary $\delta$-$J$-hyperideals as an expansion of n-ary $J$-hyperideals. Several properties of them are provided. Moreover, we extend the notion of n-ary $\delta$-$J$-hyperideals to $(k,n)$-absorbing $\delta$-$J$-hyperideals.
\section{Preliminaries}
In this section we recall some definitions and results concerning $n$-ary hyperstructures which we need to develop our paper.\\
Let $H$ be a nonempty set and $P^*(H)$ be the
set of all the non-empty subsets of $H$.  Then the mapping $f : H^n \longrightarrow P^*(H)$
 is called an $n$-ary hyperoperation and the algebraic system $(H, f)$ is called an $n$-ary hypergroupoid. For non-empty subsets $A_1,..., A_n$ of $H$ we define

$f(A^n_1) = f(A_1,..., A_n) = \bigcup \{f(x^n_1) \ \vert \  x_i \in  A_i, i = 1,..., n \}$.\\
The sequence $x_i, x_{i+1},..., x_j$ 
will be denoted by $x^j_i$. For $j< i$, $x^j_i$ is the empty symbol. Using this notation,

$f(x_1,..., x_i, y_{i+1},..., y_j, z_{j+1},..., z_n)$ \\
will be written as $f(x^i_1, y^j_{i+1}, z^n_{j+1})$. The  expression will be written in the form $f(x^i_1, y^{(j-i)}, z^n_{j+1})$, when $y_{i+1} =... = y_j = y$ . 
 
 If for every $1 \leq i < j \leq n$ and all $x_1, x_2,..., x_{2n-1} \in H$, 

$f(x^{i-1}_1, f(x_i^{n+i-1}), x^{2n-1}_{n+i}) = f(x^{j-1}_1, f(x_j^{n+j-1}), x_{n+j}^{2n-1}),$ \\
then the n-ary hyperoperation $f$ is called associative. An $n$-ary hypergroupoid with the
associative $n$-ary hyperoperation is called an $n$-ary semihypergroup. 

An $n$-ary hypergroupoid $(H, f)$ in which the equation $b \in f(a_1^{i-1}, x_i, a_{ i+1}^n)$ has a solution $x_i \in H$
for every $a_1^{i-1}, a_{ i+1}^n,b  \in H$ and $1 \leq i \leq n$, is called an $n$-ary quasihypergroup, when $(H, f)$ is an $n$-ary
semihypergroup, $(H, f)$ is called an $n$-ary hypergroup.  

An $n$-ary hypergroupoid $(H, f)$ is commutative if for all $ \sigma \in \mathbb{S}_n$, the group of all permutations of $\{1, 2, 3,..., n\}$, and for every $a_1^n \in H$ we have $f(a_1,..., a_n) = f(a_{\sigma(1)},..., a_{\sigma(n)})$.
  If  $a_1^n \in H$ then we denote $a_{\sigma(1)}^{\sigma(n)}$ as the $(a_{\sigma(1)},..., a_{\sigma(n)})$.

If $f$ is an $n$-ary hyperoperation and $t = l(n- 1) + 1$, then $t$-ary hyperoperation $f_{(l)}$ is given by

$f_{(l)}(x_1^{l(n-1)+1}) = f(f(..., f(f(x^n _1), x_{n+1}^{2n -1}),...), x_{(l-1)(n-1)+1}^{l(n-1)+1})$. 
\begin{definition}
(\cite{d1}). Let $(H, f)$ be an $n$-ary hypergroup and $B$ be a non-empty subset of $H$. $B$ is called
an $n$-ary subhypergroup of $(H, f)$, if $f(x^n _1) \subseteq B$ for $x^n_ 1 \in B$, and the equation $b \in f(b^{i-1}_1, x_i, b^n _{i+1})$ has a solution $x_i \in B$ for every $b^{i-1}_1, b^n _{i+1}, b \in B$ and $1 \leq i  \leq n$.
An element $e \in H$ is called a scalar neutral element if $x = f(e^{(i-1)}, x, e^{(n-i)})$, for every $1 \leq i \leq n$ and
for every $x \in H$. 

An element $0$ of an $n$-ary semihypergroup $(H, g)$ is called a zero element if for every $x^n_2 \in H$ we have
$g(0, x^n _2) = g(x_2, 0, x^n_ 3) = ... = g(x^n _2, 0) = 0$.
If $0$ and $0^ \prime $are two zero elements, then $0 = g(0^ \prime , 0^{(n-1)}) = 0 ^ \prime$  and so the zero element is unique. 
\end{definition}
\begin{definition}
(\cite{l1}). Let $(H, f)$ be a  $n$-ary hypergroup. $(H, f)$ is called a canonical $n$-ary
hypergroup if\\
(1) there exists a unique $e \in H$, such that for every $x \in H, f(x, e^{(n-1)}) = x$;\\
(2) for all $x \in H$ there exists a unique $x^{-1} \in H$, such that $e \in f(x, x^{-1}, e^{(n-2)})$;\\
(3) if $x \in f(x^n _1)$, then for all $i$, we have $x_i \in  f(x, x^{-1},..., x^{-1}_{ i-1}, x^{-1}_ {i+1},..., x^{-1}_ n)$.

We say that $e$ is the scalar identity of $(H, f)$ and $x^{-1}$ is the inverse of $x$. Notice that the inverse of $e$ is $e$.
\end{definition}
\begin{definition}
(\cite{d1})A Krasner $(m, n)$-hyperring is an algebraic hyperstructure $(R, f, g)$, or simply $R$,  which
satisfies the following axioms:\\
(1) $(R, f$) is a canonical $m$-ary hypergroup;\\
(2) $(R, g)$ is a $n$-ary semigroup;\\
(3) the $n$-ary operation $g$ is distributive with respect to the $m$-ary hyperoperation $f$ , i.e., for every $a^{i-1}_1 , a^n_{ i+1}, x^m_ 1 \in R$, and $1 \leq i \leq n$,

$g(a^{i-1}_1, f(x^m _1 ), a^n _{i+1}) = f(g(a^{i-1}_1, x_1, a^n_{ i+1}),..., g(a^{i-1}_1, x_m, a^n_{ i+1}))$;\\
(4) $0$ is a zero element (absorbing element) of the $n$-ary operation $g$, i.e., for every $x^n_ 2 \in R$ we have 

$g(0, x^n _2) = g(x_2, 0, x^n _3) = ... = g(x^n_ 2, 0) = 0$.
\end{definition}
We assume throughout this paper that all Krasner $(m,n)$-hyperrings are commutative.

A non-empty subset $S$ of $R$ is called a subhyperring of $R$ if $(S, f, g)$ is a Krasner $(m, n)$-hyperring. Let
$I$ be a non-empty subset of $R$, we say that $I$ is a hyperideal of $(R, f, g)$ if $(I, f)$ is an $m$-ary subhypergroup
of $(R, f)$ and $g(x^{i-1}_1, I, x_{i+1}^n) \subseteq I$, for every $x^n _1 \in  R$ and  $1 \leq i \leq n$.

\begin{definition} (\cite{sorc1}) For every element $x$ in a Krasner $(m,n)$-hyperring $R$, the hyperideal generated by $x$ is denoted by $<x>$ and defined as follows:

$<x>=g(R,x,1^{(n-2)})=\{g(r,x,1^{(n-2)}) \ \vert \ r \in R\}$
\end{definition}
\begin{definition} (\cite{sorc1})
A hyperideal $M$ of a Krasner $(m, n)$-hyperring $R$
is said to be maximal if for every hyperideal $N$ of $R$, $M \subseteq N \subseteq R$ implies that $N=M$ or $N=R$.
\end{definition}
The Jacobson radical of a Krasner $(m, n)$-hyperring $R$
is the intersection of all maximal hyperideals of $R$ and it is denoted by $J_{(m,n)}(R)$. If $R$ does not have any maximal hyperideal, we let $J_{(m,n)}(R)=R$.

\begin{definition} (\cite{sorc1})
We say that an element $x \in  R$ is invertible  if there exists $y \in R$ such that $1_R=g(x,y,1_R^{(n-2)})$. Also,
the subset $U$ of $R$ is invertible  if and only if every element of $U$ is invertible .
\end{definition}
\begin{definition}
  (\cite{sorc1})A hyperideal $P$ of a Krasner $(m, n)$-hyperring $R$, such that $P \neq R$, is called a prime hyperideal if for hyperideals $U_1,..., U_n$ of $R$, $g(U_1^ n) \subseteq P$ implies that $U_1 \subseteq P$ or $U_2 \subseteq P$ or ...or $U_n \subseteq P$.
\end{definition}
\begin{lem} 
(Lemma 4.5 in \cite{sorc1})Let $P\neq R$ be a hyperideal of a Krasner $(m, n)$-hyperring $R$. Then $P$ is a prime hyperideal if for all $x^n_ 1 \in R$, $g(x^n_ 1) \in P$ implies that $x_1 \in P$ or ... or $x_n \in P$. 
\end{lem}
\begin{definition} (\cite{sorc1}) Let $I$ be a hyperideal in a  Krasner $(m, n)$-hyperring $R$ with
scalar identity. The radical (or nilradical) of $I$, denoted by ${\sqrt I}^{(m,n)}$
is the hyperideal $\bigcap P$, where
the intersection is taken over all  prime hyperideals $P$ which contain $I$. If the set of all prime hyperideals containing $I$ is empty, then ${\sqrt I}^{(m,n)}$ is defined to be $R$.
\end{definition}
 Ameri and  Norouzi showed that if $x \in {\sqrt I}^{(m,n)}$ then 
 there exists $t \in \mathbb {N}$ such that $g(x^ {(t)} , 1_R^{(n-t)} ) \in I$ for $t \leq n$, or $g_{(l)} (x^ {(t)} ) \in I$ for $t = l(n-1) + 1$ \cite{sorc1}.
 
\begin{definition}
(\cite{sorc1}) A proper hyperideal $I$ in a  Krasner $(m, n)$-hyperring $R$ with the
scalar identity $1_R$ is said to be  primary if $g(x^n _1) \in I$ and $x_i \notin I$ implies that $g(x_1^{i-1}, 1_R, x_{ i+1}^n) \in {\sqrt I}^{(m,n)}$ for some $1 \leq i \leq n$.
\end{definition}
 If $I$ is a primary hyperideal in a  Krasner $(m, n)$-hyperring $R$ with the scalar identity $1_R$, then ${\sqrt I}^{(m,n)}$ is  prime. (Theorem 4.28 in \cite{sorc1})
\begin{definition} (\cite{sorc1})
Let $S$ be a hyperideal of a Krasner $(m, n)$-hyperring $(R, f, g)$. Then the set

$R/S = \{f(x^{i-1}_1, S, x^m_{i+1}) \ \vert \  x^{i-1}_1,x^m_{i+1} \in R \}$\\
endowed with m-ary hyperoperation $f$ which for all $x_{11}^{1m},...,x_{m1}^{mm} \in R$

$f(f(x_{11}^{1 (i-1)}, S, x^{1m}_ {1(i+1)}),..., f(x_{m1}^{ m(i-1)}, S, x^{mm}_ {m(i+1)}))$ 

$= f (f(x^{m1}_{11}),..., f(x^{m(i-1)}_{1(i-1)}), S, f(x^{m(i+1)}_{1(i+1)} ),..., f(x^{mm}_ {1m}))$\\
and with $n$-ary hyperoperation g which for all $x_{11}^{1m},...,x_{n1}^{nm} \in R$

$g(f(x_{11}^{1 (i-1)}, S, x^{1m}_ {1(i+1)}),..., f(x_{n1}^{ n(i-1)}, S, x^{nm}_ {n(i+1)}))$ 

$= f (g(x^{n1}_{11}),..., g(x^{n(i-1)}_{1(i-1)}), S, g(x^{n(i+1)}_{1(i+1)} ),..., f(x^{nm}_ {1m}))$\\
construct a Krasner $(m, n)$-hyperring, and $(R/S, f, g)$  is called the quotient Krasner $(m, n)$-hyperring of $R$ by $S$. 
\end{definition}
\begin{definition} (\cite{d1})
Let $(R_1, f_1, g_1)$ and $(R_2, f_2, g_2)$ be two Krasner $(m, n)$-hyperrings. A mapping
$h : R_1 \longrightarrow R_2$ is called a homomorphism if for all $x^m _1 \in R_1$ and $y^n_ 1 \in R_1$ we have

$h(f_1(x_1,..., x_m)) = f_2(h(x_1),...,h(x_m))$

$h(g_1(y_1,..., y_n)) = g_2(h(y_1),...,h(y_n)). $
\end{definition}
\begin{definition} (\cite{mah3})
Let $R$ be a Krasner $(m,n)$-hyperring. A function $\delta$  is called a hyperideal expansion of $R$  if  it assigns to each hyperideal $I$ of $R$  a hyperideal $\delta(I)$  of $R$ with   the following conditions:

$(i)$ $I \subseteq \delta(I)$.

$(ii)$ if $I \subseteq J$ for any hyperideals $I, J$ of $R$, then $\delta (I) \subseteq \delta (J)$ . 
\end{definition}
For example, let $R$ be a Krasner $(m,n)$-hyperring.

1. Define $\delta_0(I)=I$, for each  hyperideal $I$ of $R$. Then $\delta_0$ is a hyperideal expansion of $R$.

2. Define $\delta_1(I)=\sqrt{I}^{(m,n)}$,  for each  hyperideal $I$ of $R$. Then $\delta_1$ is a hyperideal expansion of $R$.

3. Define $\delta_R(I)=R$,   for each  hyperideal $I$ of $R$. Then $\delta_R$ is a hyperideal expansion of $R$.

4. Define $\delta_q(I/J)=\delta(I)/J$, for each  hyperideal $I$ of $R$ containing hyperideal $J$ and expansion function $\delta$ of $R$. Then $\delta_q$ is  a hyperideal expansion  of $R/J$.
\begin{definition}(\cite{mah3})
Let $(R_2,f_2,g_2)$ and  $(R_2,f_2,g_2)$ be two Krasner $(m,n)$-hyperrings and $h: R_1 \longrightarrow R_2$ a  hyperring homomorphism. Let $\delta$ and $\gamma$ be hyperideal expansions   of $R_1$ and
$R_2$, respectively. Then $h$ is said to be a $\delta \gamma$-homomorphism if $\delta(h^{-1}(I_2)) = h^{-1}(\gamma(I_2))$ for  hyperideal $I_2$ of
$R_2$. 
\end{definition}
 Note that $\gamma(h(I_1)=h(\delta(I_1)$ for   $\delta \gamma$-epimorphism $h$ and   for hyperideal $I_1$ of $R_1$ with $ Ker (h) \subseteq I_1$.
For example, let $(R_1,f_1,g_1)$ and $(R_2,f_2,g_2)$ be two Krasner $(m,n)$-hyperrings. If $\delta_1$ of $R_1$ and $\gamma_1$ of $R_2$ be the hyperideal expansions defined in Example 3.2, in (\cite{mah3}), then each homomorphism $h:R_1 \longrightarrow R_2$ is a $\delta_1 \gamma_1$-homomorphism.
\section{$n$-ary $J$-hyperideals}
\begin{definition} 
A hyperideal proper $Q$ of  a Krasner $(m,n)$-hyperring with the scalar identity $1_R$ is said to be n-ary $J$-hyperideal if whenever  $x_1^n \in R$ with $g(x_1^n) \in Q$ and $x_i  \notin J_{(m,n)}(R)$ implies that $g(x_1^{i-1},1_R,x_{i+1}^n) \in Q$.
\end{definition}
\begin{example}
The set $A=\{0,1,x\}$ with the following 3-ary hyperoeration$f$ and 3-ary operation $g$ is a Krasner $(3,3)$-hyperring such that $f$ and $g$ are commutative.
\[f(0,0,0)=0, \ \ \ f(0,0,1)=1, \ \ \ f(0,1,1)=1, \ \ \ f(1,1,1)=1, \ \ \ f(1,1,x)=A\]
\[f(0,1,x)=A, \ \ \ f(0,0,x)=x,\ \ \ f(0,x,x)=x,\ \ \ f(1,x,x)=A, \ \ \ f(x,x,x)=x\]
$\ \ \ g(1,1,1)=1,\ \ \ \ g(1,1,x)=g(1,x,x)=g(x,x,x)=x$\\

and for $x_1,x_2 \in R, g(0,x_1,x_2)=0$.
In the Krasner $(3,3)$-hyperring,  hyperideals $\{0 \}$ and $\{0,x\}$ are two n-ary $J$-hyperideals of $A$. 
\end{example}
\begin{example}
The set $R=\{0,1,\alpha,\beta\}$ whih following 2-hyperoperation $"\oplus"$ is a canonical 2-ary hypergroup.

\hspace{1.5cm}
\begin{tabular}{c|c} 
$\oplus$ & $0$  \ \ \  \  \ \ \  $1$  \ \ \ \ \ \ \ $\alpha$  \ \ \ \ \ \ \ $\beta$
\\ \hline 0 & $0$\ \ \  \  \ \ \ $1$\ \ \ \ \ \ \  \ \ $\alpha$ \ \ \ \ \ \ \ $\beta$ 
\\ $1$ & $1$ \ \ \  \  \ \ \  $A$ \ \ \ \ \  \  \ $\beta$ \ \ \ \ \ \ $B$
\\ $\alpha$ & $\alpha$ \ \ \  \  \ \ \  $\beta$ \ \ \ \ \ \ \ $0$ \ \ \ \ \ \ \ $1$
\\ $\beta$ & $\beta$ \ \ \  \  \ \ \  $B$ \ \ \ \ \ \ \ $1$ \ \ \ \ \ \ \ $A$
\end{tabular}

In which $A=\{0,1\}$ and $B=\{\alpha,\beta\}$. Define $g(a_1^n)=\alpha$ for $a_1,a_2,a_3,a_4 \in B$,  otherwise, $g(a_1^n)=0$.
It follows that $(R,\oplus,g)$ ia a Krasner (2,4)-hyperring. In the hyperring, $\{0\}$ is a 4-ary $J$-hyperideal.
\end{example}
\begin{theorem} \label{34}
Let $Q$ be an n-ary $J$-hyperideal of a Krasner $(m,n)$-hyperring $R$. Then $Q \subseteq J_{(m,n)}(R)$.
\end{theorem}
\begin{proof}
Let $Q$ be an n-ary $J$-hyperideal of a Krasner $(m,n)$-hyperring $R$ such that $Q \nsubseteq J_{(m,n)}(R)$. Suppose that  $x \in Q$ but $x \notin J_{(m,n)}(R)$. Since $Q$ is a n-ary $J$-hyperideal of $R$ and $g(x,1_R^{(n-1)}) \in Q$, then we have $g(1_R^{(n)}) \in Q$ which is a contradiction. Therefore, $Q \subseteq J_{(m,n)}(R)$.
\end{proof}
Next, we characterize the Krasner ${(m,n)}$-hyperring  which every proper hyperideal is an n-ary $J$-hyperideal.

\begin{theorem} \label{31}
Let $R$ be a Krasner ${(m,n)}$-hyperring.
Then $R$ is  local 
if and only if every proper hyperideal of $R$ is an n-ary $J$-hyperideal.  
\end{theorem}
\begin{proof}
$ \Longrightarrow$ Let  $M$ is the only maximal hyperideal of $R$. Then $J_{(m,n)}(R)=M$. Suppose that $Q$ is a proper hyperideal of $R$. Let for $x_1^n \in R$, $g(x_1^n) \in Q$ such that $x_i \notin M$. Therefore  $x_i$ is invertible. Then we have 
\begin{align*}
g(x_i^{-1},g(x_1^n),1_R^{(n-2)})&=g(g(x_i,x_i^{-1},1_R^{(n-2)}),g(x_1^{i-1},1_R,x_{i+1}^n),1_R^{(n-2)})\\
&=g(x_1^{i-1},1_R,x_{i+1}^n)\\
& \subseteq Q
\end{align*}
Hence, $Q$ is a an n-ary $J$-hyperideal of $R$. 

$\Longleftarrow$ Let every proper hyperideal of $R$ is an n-ary $J$-hyperideal. Assume that the hyperideal $M$ of $R$ is  maximal. Let $x \in M$. By the hypothesis, the principal hyperideal $\langle x \rangle$ is an n-ary $J$-hyperideal of $R$. Since  $g(x,1_R^{(n-1)}) \in \langle x \rangle$, then we get $x \in J_{(m,n)}(R)$ or $g(1_R^{(n)}) \in \langle x \rangle$. Since  the second case is  a contradiction, then $x \in J_{(m,n)}(R)$ which implies $J_{(m,n)}(R)=M$. Consequently, $R$ is a local Krasner $(m,n)$-hyperring.
\end{proof}

\begin{theorem} \label{entersection}
Let $\{Q_i\}_{i \in \Delta}$ be a  nonmepty set of n-ary $J$-hyperideals of  a Krasner $(m,n)$-hyperring $R$. Then $\bigcap_{i \in \Delta}Q_i$ is an n-ary $J$-hyperideal of $R$.
\end{theorem}
\begin{proof}
Let $g(x_1^n) \in  \bigcap_{i \in \Delta}Q_i$ for some $x_1^n \in R$ such that $x_i \notin J_{(m,n)}(R)$. Then $g(x_1^n) \in Q_i$ for every $i \in \Delta$. Since $Q_i$ is an n-ary $J$-hyperideal of $R$, we have $g(x_1^{i-1},1_R,x_{i-1}^n) \in Q_i$. Then $g(x_1^{i-1},1_R,x_{i-1}^n) \in \bigcap_{i \in \Delta}Q_i$.
\end{proof}
\begin{theorem} \label{36}
Let $Q$ be a proper hyperideal of  a Krasner $(m,n)$-hyperring $R$. Then the following statements are equivalent:

(1) $Q$ is an n-ary $J$-hyperideal of $R$.

(2) $Q=U_x$ where $U_x=\{y \in R \ \vert \ g(x,y,1_R^{(n-2)}) \in Q\}$  for every $x \notin J_{(m,n)}(R)$.

(3) $g(I_1^n) \subseteq Q$ for some hyperideals $I_1^n$  of $R$ and $I_i \nsubseteq J_{(m,n)}(R)$ imply that $g(I_1^{i-1},1_R,I_{i+1}^n) \subseteq Q$.

\end{theorem}
\begin{proof}
$(1) \Longrightarrow (2)$ Let $Q$ be an n-ary  $J$-hyperideal of $R$.  We have  $Q \subseteq U_x$ for every $x \in R$. Suppose that $y \in U_x$ such that $x \notin J_{(m,n)}(R)$. This means $g(x,y,1_R^{(n-2)}) \in Q$. Since $Q$ is an n-ary $J$-hyperideal of $R$ and $x \notin J_{(m,n)}(R)$, then $y=g(y,1_R^{(n-2)}) \in Q$. Hence, we get $Q=U_x$.

$(2) \Longrightarrow (3)$ Let $g(I_1^n) \subseteq Q$ for some hyperideals $I_1^n$  of $R$ such that  $I_i \nsubseteq J_{(m,n)}(R)$. Take $x_i \in I_i$ such that $x_i \notin J_{(m,n)}(R)$. Hence, $g(I_1^{i-1},x_i, I_{I+1}^n) \subseteq Q$ which means $g(I_1^{i-1},1_R,I_{i+1}^n) \subseteq U_{x_i}$. Since $Q=U_{x_i}$ for every $x_i \notin J(R)$, then $g(I_1^{i-1},1_R,I_{i+1}^n) \subseteq Q$.

$(iii) \Longrightarrow (i)$ Let $g(x_1^n) \in  Q$ for some $x_1^n \in R$ with $x_i \notin J_{(m,n)}(R)$. We have $g(\langle x_1 \rangle,..., \langle x_n \rangle)=g( \langle g(x_1^n)  \rangle, 1_R^{(n-1)}) \subseteq Q$ but $\langle x_i \rangle \nsubseteq J_{(m,n)}(R)$. Then we get $g(\langle x_1 \rangle,..., \langle x_{i-1} \rangle,1_R, \langle x_{i+1} \rangle,...,\langle x_n \rangle)=g(\langle g(x_1^{i-1},1_R,x_{i+1}^n)\rangle,1^{(n-1)})\in Q$ which implies $g(x_1^{i-1},1_R,x_{i+1}^n) \in Q$. Therefore, $Q$ is an n-ary $J$-hyperideal of $R$.
\end{proof}
\begin{theorem}
Let $Q$ be a proper hyperideal of a Krasner $(m,n)$-hyperring $R$. Then $Q$ is an n-ary $J$-hyperideal of $R$ if and only if $U_x \subseteq J_{(m,n)}(R)$ where $U_x=\{y \in R \ \vert \ g(x,y,1_R^{(n-2)}\in Q\}$ for every $x \notin Q$.
\end{theorem}
\begin{proof}
$\Longrightarrow$ Let $y \in U_x$ such that $x\notin Q$. So, $g(x,y,1^{(n-2)}) \in Q$. Then we have $y \in J(R)$ as $Q$ is an n-ary  $J$-hyperideal of $R$ and $x=g(x,1^{(n-2)}) \notin Q$ .

$\Longleftarrow $ Let $g(x_1^n) \in Q$ for some $x_1^n \in R$ such that $x_i \notin J_{(m,n)}(R)$. If $g(x_1^{i-1},1_R,x_{i+1}^n) \notin Q$, then $x_i \in U_{g(x_1^{i-1},1_R,x_{i+1}^n)} \subseteq J_{(m,n)}(R)$ which is a contradiction. Then we conclude that  $g(x_1^{i-1},1_R,x_{i+1}^n) \in Q$. Thus, $Q$ is an n-ary $J$-hyperideal of $R$.
\end{proof}
\begin{theorem} \label{33}
Let $Q$ be a hyperideal of a Krasner $(m,n)$-hyperring $R$ and  $S$ be a nonempty subset of $R$ such that $S \nsubseteq Q$. If $Q$ is an n-ary $J$-hyperideal of $R$, then $U_S=\{x \in R \ \vert \ g(x,S,1_R^{(n-2)}) \in Q\}$ is an n-ary $J$-hyperideal of $R$.
\end{theorem}
\begin{proof}
Let $U_S=R$. Then $1_R \in U_S$ which implies $S \subseteq Q$ a contradiction. Hence, $U_S$ is a proper hyperideal of $R$. Suppose that $g(x_1^n) \in U_S$ for some $x_1^n \in R$ such that $x_i \notin J_{(m,n)}(R)$. This means $g(g(x_1^n),S,1_R^{(n-2)})=\bigcup_{s \in S}g(g(x_1^n),s,1^{(n-2)}) \subseteq Q$ which implies for each $s \in S$, $g(g(x_1^n),s,1_R^{(n-2)})=g(g(x_1^{i-1},s,x_{i+1}^n),x_i,1_R^{(n-2)}) \in Q$. Then $g(x_1^{i-1},s,x_{i+1}^n)=g(g(x_1^{i-1},s,x_{i+1}^n),1^{(n-1)}) \in Q$ for all $s \in S$ as $Q$ is an n-ary $J$-hyperideal of $R$. This means $g(x_1^{i-1},1_R,x_{i+1}^n) \in U_S$. Thus, $U_S$ is an n-ary $J$-hyperideal of $R$.
\end{proof}
\begin{theorem} \label{35}
Let $Q$ be an n-ary $J$-hyperideal of a Krasner $(m,n)$-hyperring $R$ such that there is no n-ary $J$-hyperideal which contains $Q$ properly. Then $Q$ is an n-ary prime hyperideal of $R$. 
\end{theorem}
\begin{proof}
Assume that $Q$ is an n-ary $J$-hyperideal of a Krasner $(m,n)$-hyperring $R$ such that there is no n-ary $J$-hyperideal which contains $Q$ properly. Suppose that $g(x_1^n) \in Q$ for some $x_1^n \in R$ such that for every $1 \leq i \leq n-1$, $x_i \notin Q$. Put $S=\{x_1,...,x_{n-1}\}$. By Theorem \ref{33}, $(Q:S)$ is an n-ary  $J$-hyperideal of $R$. Since $Q \subseteq U_S=\{x \in R \ \vert \ g(x,S,1^{(n-2)}) \in Q\}$, we conclude that $x_n \in (Q:S)=Q$, by the hypothesis. Thus, $Q$ is a prime hyperideal.
\end{proof}
In Theorem \ref{35}, if $Q=J_{(m,n)}(R)$, then the inverse of the theorem is true. 
\begin{theorem}
Let $J_{(m,n)}(R)$  be an n-ary prime of $R$. Then $J_{(m,n)}(R)$ is an n-ary  $J$-hyperideal of $R$ such that there is no $J$-hyperideal which contains $J_{(m,n)}(R)$ properly.
\end{theorem}
\begin{proof}
Put $Q=J_{(m,n)}(R)$ such that $J_{(m,n)}(R)$  is an n-ary prime of $R$. Let $g(x_1^n) \in Q$ for some $x_1^n \in R$ with $x_i \notin J(R)$. Since $Q$ is an n-ary prime hyperideal of $R$, then there exists $1 \leq j \leq i-1$ or $i+1 \leq j \leq n$ such that $x_j \in Q=J_{(m,n)}(R)$ which means the hyperideal $J(R)$ of $R$ is an n-ary $J$-hyperideal. By Theorem \ref{34},  there is no $J$-prime hyperideal which contains $Q$ properly.
\end{proof}
\section{$n$-ary $\delta$-$J$-hyperideals}
In this section, we define and study the concept of $n$-ary $\delta$-$J$-hyperideals as an expansion of $n$-ary $J$-hyperideals. 
\begin{definition}
Let $\delta$ be a hyperideal expansion of a Krasner $(m,n)$-hyperring $R$. A proper hyperideal $Q$ of $R$ is called  n-ary $\delta$-$J$-hyperideal if for $x_1^n \in R$, $g(x_1^n) \in Q$ implies that $x_i \in J_{(m,n)}(R)$ or $g(x_1^{i-1},1_R,x_{i+1}^n) \in \delta(Q).$
\end{definition}
\begin{example} 
The hyperideal $I=\{0\mathbb{Z}^\star_{12},4\mathbb{Z}^\star_{12}\}$ in $\mathbb{Z}_{12}/\mathbb{Z}_{12}^\star$ of Example 4.1 in \cite{rev2} is a $\delta_1$-$J$-hyperideal.
\end{example}
\begin{theorem} \label{0041}
Let $Q$ be a proper hyperideal of a Krasner $(m,n)$-hyperring $R$. If $\delta(Q)$ is an n-ary $J$-hyperideal of $R$, then $Q$ is an n-ary $\delta$-$J$-hyperideal of $R$.
\end{theorem}
\begin{proof}
Suppose that $\delta(Q)$ is an n-ary $J$-hyperideal of $R$. Let $g(x_1^n) \in Q$ for some $x_1^n \in R$ such that $x_i \notin J_{(m,n)}(R)$. Since $\delta(Q)$ is an n-ary $J$-hyperideal of $R$ and $Q \subseteq \delta(Q)$, then $g(x_1^{i-1},1_R,x_{i+1}^n) \in \delta(Q)$ which implies $Q$ is a $\delta$-$J$-hyperideal of $R$.
\end{proof}
The next Theorem shows that the inverse of Theorem \ref{0041} is true  $\delta=\delta_1$.
\begin{theorem}
Let $Q$ be a proper hyperideal of a Krasner $(m,n)$-hyperring $R$. If $Q$ is an n-ary $\delta_1$-$J$-hyperideal of $R$, then $\sqrt{Q}^{(m,n)}$ is an n-ary $J$-hyperideal of $R$.
\end{theorem}
\begin{proof}
Let for $x_1^n \in R$, $g(x_1^n)\in \sqrt{Q}^{(m,n)}$ such that $x_i \notin J_{(m,n)}(R)$. By $g(x_1^n)\in \sqrt{Q}^{(m,n)}$, it follows that there exists $t \in \mathbb{N}$ such that if $t \leq n$, then $g(g(x_1^n)^{(t)},1_R^{(n-t)}) \in Q$. Hence by associativity we get

$g(x_i^{(t)},g(x_i^{i-1},1_R,x_{i+1}^n)^{(t)},1_R^{(n-2t)})$

$=g(x_i^{(t)},g(x_i^{i-1},1_R,x_{i+1}^n)^{(t)},g(1_R^{(n)}),1_R^{(n-2t-1)})$

$=g(g(x_i^{(t)},1_R^{(n-t)}),g(x_i^{i-1},1_R,x_{i+1}^n)^{(t)},1_R^{(n-t-1)})$

$\subseteq Q$.\\
Since $Q$ is an n-ary $\delta_1$-$J$-hyperideal of $R$, then 

$g(x_i^{(t)},1^{(n-t)}) \in J_{(m,n)}(R)$ $\hspace{1cm}$\\
or

$g(g(x_i^{i-1},1_R,x_{i+1}^n)^{(t)},1_R^{(n-t)}) \in \delta_1(Q)=\sqrt{Q}^{(m,n)}$.\\
If $g(x_i^{(t)},1^{(n-t)}) \in J_{(m,n)}(R)$, then $x_i \in \sqrt{J_{(m,n)}(R)}^{(m,n)}=J_{(m,n)}(R)$ which is a contradiction. Then we have 

 $g(g(x_i^{i-1},1_R,x_{i+1}^n)^{(t)},1_R^{(n-t)}) \in \sqrt{Q}^{(m,n)}$\\
 which means $g(x_i^{i-1},1_R,x_{i+1}^n) \in \sqrt{Q}^{(m,n)}$. Thus we conclude that $\sqrt{Q}^{(m,n)}$ is an n-ary  $J$-hyperideal of $R$.
  If $t=l(n-1)+1$, then by using a similar argument, one can easily complete the proof.
\end{proof}
\begin{theorem}
Let $Q$ be a proper hyperideal of a Krasner $(m,n)$-hyperring $R$ and let   $\delta$ and $\gamma$ be two hyperideal expansions of $R$. If $\delta(Q)$ is an n-ary $\gamma$-$J$-hyperideal of $R$, then $Q$ is an n-ary $\gamma \circ \delta$-$J$-hyperideal of $R$.
\end{theorem}
\begin{proof}
Suppose that $\delta(Q)$ is an n-ary $\gamma$-$J$-hyperideal of $R$. Let $g(x_1^n) \in Q$ for some $x_1^n \in R$ such that $x_i \notin J_{(m,n)}(R)$. We get  $g(x_1^n) \in \delta(Q)$ as $Q \subseteq \delta(Q)$. Since $\delta(Q)$ is an n-ary $\gamma$-$J$-hyperideal of $R$ and $x_i \notin J_{(m,n)}(R)$, then $g(x_1^{i-1},1_R,x_{i+1}^n) \in \gamma (\delta(Q))$ which means $g(x_1^{i-1},1_R,x_{i+1}^n) \in \gamma \circ \delta(Q)$. Thus,  $Q$ is an n-ary $\gamma \circ \delta$-$J$-hyperideal of $R$.
\end{proof}
\begin{theorem}
 Let $Q_1,Q_2$ and $Q_3$ be three proper hyperideals of a Krasner $(m,n)$-hyperring $R$ such that $Q_1 \subseteq Q_2 \subseteq Q_3$. If $Q_3$ is an n-ary $\delta$-$J$-hyperideal of $R$ and $\delta(Q_1)=\delta(Q_3)$, then $Q_2$ is an n-ary $\delta$-$J$-hyperideal of $R$.
 \end{theorem}
 \begin{proof}
 Let $g(x_1^n) \in Q_2$ for some $x_1^n \in R$ such that $x_i \notin J_{(m,n)}(R)$. Since $Q_2 \subseteq Q_3$ and $Q_3$ is an n-ary $\delta$-$J$-hyperideal of $R$, then we have $g(x_1^{i-1},1_R,x_{i+1}^n) \in \delta(Q_3)$.  Then we conclude that $g(x_1^{i-1},1_R,x_{i+1}^n) \in \delta(Q_1)$, by the hypothesis. Since $Q_1 \subseteq Q_2$, then $\delta(Q_1) \subseteq \delta(Q_2)$. This implies that $g(x_1^{i-1},1_R,x_{i+1}^n) \in \delta(Q_2)$ as needed.
 \end{proof} 
 \begin{theorem}
Let $Q$ be a $\delta$-$J$-hyperideal of a Krasner $(m,n)$-hyperring $R$ such that  $\sqrt{\delta(Q)}^{(m,n)} \subseteq  \delta (\sqrt{Q}^{(m,n)})$. Then $\sqrt{Q}^{(m,n)}$ is a $\delta$-$J$-hyperideal of $R$. 
\end{theorem}
\begin{proof}
Let for $x_1^n \in R$, $g(x_1^n)\in \sqrt{Q}^{(m,n)}$ such that $x_i \notin J_{(m,n)}(R)$. By $g(x_1^n)\in \sqrt{Q}^{(m,n)}$, it follows that there exists $t \in \mathbb{N}$ such that if $t \leq n$, then $g(g(x_1^n)^{(t)},1_R^{(n-t)}) \in Q$. Hence by associativity we get

$g(x_i^{(t)},g(x_i^{i-1},1_R,x_{i+1}^n)^{(t)},1_R^{(n-2t)})$

$=g(x_i^{(t)},g(x_i^{i-1},1_R,x_{i+1}^n)^{(t)},g(1_R^{(n)}),1_R^{(n-2t-1)})$

$=g(g(x_i^{(t)},1_R^{(n-t)}),g(x_i^{i-1},1_R,x_{i+1}^n)^{(t)},1_R^{(n-t-1)})$

$\subseteq Q$.\\
Since $Q$ is an n-ary $\delta$-$J$-hyperideal of $R$, then
 
$g(x_i^{(t)},1^{(n-t)}) \in J_{(m,n)}(R)$ 
or $\hspace{1cm}$ $\hspace{1cm}$
$g(g(x_i^{i-1},1_R,x_{i+1}^n)^{(t)},1_R^{(n-t)}) \in \delta(I)$.\\
If $g(x_i^{(t)},1^{(n-t)}) \in J_{(m,n)}(R)$, then $x_i \in \sqrt{J_{(m,n)}(R)}^{(m,n)}=J_{(m,n)}(R)$ which is a contradiction. Then we have 

 $g(g(x_i^{i-1},1_R,x_{i+1}^n)^{(t)},1_R^{(n-t)}) \in \delta(Q)$\\
 which means $g(x_i^{i-1},1_R,x_{i+1}^n) \in \sqrt{\delta(Q)}^{(m,n)}$. By the assumption, we obtain
 
  $g(x_i^{i-1},1_R,x_{i+1}^n) \in \delta(\sqrt{Q}^{(m,n)})$.  \\Thus we conclude that $\sqrt{Q}^{(m,n)}$ is a $\delta$-$J$-hyperideal of $R$.
  If $t=l(n-1)+1$, then by using a similar argument, one can easily complete the proof.
\end{proof}
We say that $\delta$  has the property of intersection preserving if it satisfies $\delta(I \cap J)=\delta(I) \cap \delta(J)$, for all hyperideals $I, J$ of $R$. For example, 
the hyperideal expansion  $\delta_1$  of a Krasner $(m,n)$-hyperring $R$ has the property of intersection preserving.
\begin{theorem}
Suppose that $Q_1^n$ are n-ary $\delta$-$J$-hyperideal of a Krasner $(m,n)$-hyperring $R$ and   the hyperideal expansion  $\delta$ of $R$ has the property of intersection preserving. Then $Q =\bigcap_{i=1}^n Q_i$ is an n-ary $\delta$-$J$-hyperideal of $R$.
\end{theorem}
\begin{proof}
Let $g(x_1^n) \in Q$ for some $x_1^n \in R$ such that $g(x_1^{i-1},1_R,x_{i+1}^n) \notin \delta (Q)$. Since the hyperideal expansion  $\delta$ of $R$ has the property of intersection preserving, then there exists $1 \leq j \leq n$ such that $g(x_1^{i-1},1_R,x_{i+1}^n) \notin \delta (Q_j)$. Thus, we get $x_i \in J_{(m,n)}(R)$ as $Q_j$ ia an n-ary $\delta$-$J$-hyperideal of $R$. Consequently, $Q =\bigcap_{i=1}^n Q_i$ is an n-ary $\delta$-$J$-hyperideal of $R$.
\end{proof}
\begin{theorem} \label{41}
Let $Q$ be a proper hyperideal of a Krasner $(m,n)$-hyperring $R$. Then  the following 
are equivalent:

(1) $Q$ is an n-ary $\delta$-$J$-hyperideal of $R$.

(2) If $Q_1^{n-1}$ are some hyperideals  of $R$ and $x \in R$ such that $g(Q_1^{n-1},x) \subseteq Q$, then $x \in J_{(m,n)}(R)$ or $g(Q_1^{n-1},1_R) \subseteq \delta(Q)$.

(3) If $Q_1^n$ are some hyperideals of $R$ and  $g(Q_1^n) \subseteq Q$, then $Q_i \subseteq J_{(m,n)}(R)$ or  $g(Q_1^{i-1},1_R,Q_{i+1}^n) \subseteq \delta(Q).$
\end{theorem}
\begin{proof}
$(1) \Longrightarrow (2)$ Let $Q$ be an n-ary $\delta$-$J$-hyperideal of $R$. Assume that  $g(Q_1^{n-1},x) \subseteq Q$ for some hyperideals $Q_1^{n-1}$ of $R$ such that $x \notin J_{(m,n)}(R)$. Therefore, for each $q_i \in Q_i$ with $1 \leq i \leq n-1$ we have $g(q_1^{n-1},x) \in Q$. Since $Q$ is an n-ary $\delta$-$J$-hyperideal of $R$ and $x \notin J_{(m,n)}(R)$, then $g(q_1^{n-1},1_R) \in \delta(Q)$ which means $g(Q_1^{n-1},1_R) \subseteq Q$.

$(2) \Longrightarrow (3)$ Let $g(Q_1^n) \subseteq Q$ for some hyperideals $Q_1^n$ of $R$ such that $Q_i \nsubseteq J_{(m,n)}(R)$. Take $x \in Q_i$ but $x \notin J_{(m,n)}(R)$. Since $g(Q_1^{i-1},x,Q_{i+1}^n) \subseteq Q$ and $x \notin J_{(m,n)}$, then $g(Q_1^{i-1},1_R,Q_{i+1}^n) \subseteq \delta(Q)$, by the hypothesis.

 $(3) \Longrightarrow (1)$ Let  $g(x_1^n) \in Q$ for some  $x_1^n \in R$ but $x_i \notin J_{(m,n)}(R)$. Therefore $g(\langle x_1 \rangle,...,\langle x_n\rangle) \subseteq Q$ but $\langle x_i\rangle \nsubseteq J_{(m,n)}$. Thus we have 
 
 $g(\langle x_1 \rangle,...,\langle x_{i-1} \rangle,1_R,\langle x_{i+1} \rangle ,...,\langle x_n \rangle) \subseteq \delta(Q)$. \\This means $g(x_1^{i-1},1_R,x_{i+1}^n) \in \delta(Q)$.
 \end{proof}
 \begin{theorem} \label{42}
 Let $Q$ be a proper hyperideal of a Krasner $(m,n)$-hyperring $R$. Then  the following  are equivalent:
 
 (1) $Q$ is an n-ary $\delta$-$J$-hyperideal  of $R$.
 
 (2) $Q \subseteq J_{(m,n)}(R)$ and if $g(x_1^n) \in Q$ for some $x_1^n \in R$, then either $x_i$ is in the intersection of all maximal hyperideals of $R$ containing $Q$ or $g(x_1^{i-1},1_R,x_{i+1}^n) \in \delta(Q)$. 
 \end{theorem}
 \begin{proof}
 $(1) \Longrightarrow (2)$ Suppose that $Q$ is an n-ary $\delta$-$J$-hyperideal  of $R$. Let $Q \nsubseteq J_{(m,n)}(R)$. Take $x \in Q$ such that $x \notin  J_{(m,n)}(R)$. Since $g(x,1^{(n-1)}) \in Q$, then $g(1^{(n)}) \in Q$, a contradiction. Hence, $Q \subseteq J_{(m,n)}(R)$. Since $J_{(m,n)}(R)$ is in  the intersection of all maximal hyperideals of $R$ containing $Q$, then the second assertion follows.
 
 $(2) \Longrightarrow (1)$ Let $g(x_1^n) \in Q$ for some $x_1^n \in R$ such that $x_i \notin J_{(m,n)}(R)$. The intersection of all maximal hyperideals of $R$ containing $Q$ is in $J_{(m,n)}(R)$ as $Q \subseteq J_{(m,n)}(R)$. This means $x_i$ is not in the intersection of all maximal hyperideals of $R$ containing $Q$. Then $g(x_1^{i-1},1_R,x_{i+1}^n) \in \delta(Q)$. Consequently, $Q$ is an n-ary $\delta$-$J$-hyperideal  of $R$.
 \end{proof}
\begin{theorem} \label{43}
Let $R$ be a Krasner $(m,n)$-hyperring. Then  the following  are equivalent:

(1) $R$ is a local Krasner $(m,n)$-hyperring such that $J_{(m,n)}(R)$ is the only maximal hyperideal of $R$.

(2) Every proper principal hyperideal is an n-ary  $\delta$-$J$-hyperideal  of $R$.

(3) Every proper hyperideal is an n-ary  $\delta$-$J$-hyperideal  of $R$.
\end{theorem}
\begin{proof}
 $(1) \Longrightarrow (2)$ Let $R$ be a local Krasner $(m,n)$-hyperring such that $J_{(m,n)}(R)$ is the only maximal hyperideal of $R$. Consider the principal hyperideal $Q=\langle a \rangle$ for some element $a$ of $R$ which is not invertible. Suppose that $g(x_1^n) \in Q$ such that $x_i \notin J_{(m,n)}(R)$. This means $x_i$ is  invertible  member of $R$. Then we have $g(g(x_1^n),x_i^{-1},1_R^{(n-2)})=g(g(x_1^{i-1},1_R,x_{i+1}^n),x_i,x_i^{-1},1_R^{(n-3)}) = g(x_1^{i-1},1_R,x_{i+1}^n) \subseteq Q \subseteq \delta(Q)$. Thus, $Q$ is an n-ary $\delta$-$J$-hyperideal of $R$.

$(2) \Longrightarrow (3)$ Suppose that $P$ is a proper hyperideal  of $R$ and $g(x_1^n) \in P$ for some $x_1^n \in R$ such that $x_i \notin J_{(m,n)}(R)$.  Consider the principal hyperideal $Q=\langle g(x_1^n) \rangle$. By the assumption, $Q$ is an n-ary  $\delta$-$J$-hyperideal  of $R$. Since  $ g(x_1^n) \in Q$ and $x_i \notin J_{(m,n)}(R)$, then $g(x_1^{i-1},1_R,x_{i+1}^n) \in \delta(Q)$. Since $Q \subseteq P$, then $\delta(Q) \subseteq \delta(P)$ and so $g(x_1^{i-1},1_R,x_{i+1}^n) \in \delta(P)$. This completes the proof.

 $(3) \Longrightarrow (1)$ Suppose that $M$ is a maximal hyperideal of $R$. By the hypothesis, $M$ is an n-ary $\delta$-$J$-hyperideal  of $R$. By theorem \ref{42}, we have $M \subseteq J_{(m,n)}(R)$. Since $J_{(m,n)}(R) \subseteq M$, then $M = J_{(m,n)}(R)$. This implies that $J_{(m,n)}(R)$ is the only maximal hyperideal of $R$. Thus, $R$ is  a local Krasner $(m,n)$-hyperring.
\end{proof}
Recall that a  proper hyperideal $I$ of $R$ is said to be  $\delta$-primary if
for all $x_1^n \in R$, $g(x_1^n) \in I$ implies that $x_i \in I$ or $g(x_1^{i-1},1_R,x_{i+1}^n) \in \delta(I)$ for some $1 \leq i \leq n$ \cite{mah3}.
\begin{theorem} \label{44}
Let $Q$ be a $\delta$-primary hyperideal of $R$. Then $Q$ is an n-ary $\delta$-$J$-hyperideal of $R$ if and only if $Q \subseteq J_{(m,n)}(R)$.
\end{theorem}
\begin{proof}
$\Longrightarrow$ It follows by \ref{42}. 

$\Longleftarrow$ Let $Q \subseteq J_{(m,n)}(R)$. Suppose that $g(x_1^n) \in Q$ for some $x_1^n \in R$ such that $x_i \notin J_{(m,n)}(R)$. Then $x_i \notin Q$. Thus, we get $g(x_1^{i-1},1_R,x_{i+1}^n) \in \delta(Q)$ as $Q$ is a $\delta$-primary hyperideal of $R$. This implies that $Q$ is an n-ary $\delta$-$J$-hyperideal of $R$.
\end{proof}
\begin{theorem}
Let $Q$ be a maximal hyperideal of $R$. Then $Q$ is an n-ary $\delta$-$J$-hyperideal of $R$ if and only if $Q = J_{(m,n)}(R)$.
\end{theorem}
\begin{proof}
$\Longrightarrow$ Let $Q$ be an n-ary $\delta$-$J$-hyperideal of $R$. Since $Q$ is a maximal hyperideal of $R$, then $J_{(m,n)}(R) \subseteq Q$. Let $g(x_1^n) \in Q$ for some $x_1^n \in R$ such that $x_i \notin Q$. Then $x_i \notin J_{(m,n)}(R)$. Since $Q$ is an n-ary $\delta$-$J$-hyperideal of $R$ and $x_i \notin J_{(m,n)}(R)$, we conclude that $g(x_1^{i-1},1_R,x_{i+1}^n) \in \delta(Q)$. This implies that $Q$ is a $\delta$-primary hyperideal of $R$. Therefore, we get $Q \subseteq J_{(m,n)}(R)$, by Theorem \ref{44}. Hence $Q = J_{(m,n)}(R)$.

$\Longleftarrow$ It is obvious.
\end{proof}

\begin{theorem} \label{45}
Let $(R_1,f_1,g_1)$ and $(R_2,f_2,g_2)$ be two Krasner $(m,n)$-hyperrings and $ h:R_1 \longrightarrow R_2$ be a $\delta \gamma$-homomorphism such that $\delta$ and  $\gamma$ be hyperideal expansions of $R_1$ and $R_1$, respectively. Then the following statements hold :

$(1)$ If $h$ is a monomorphism and $I_2$ is an n-ary  $\gamma$-$J_2$-hyperideal of $R_2$, then $h^{-1} (I_2)$ is an n-ary  $\delta$-$J_1$-hyperideal of $R_1$.

 $(2)$ Let $h $ be an epimorphism and $I_1$ be a hyperideal of $R$ such that  $Ker(h) \subseteq I_1$. If  $I_1$ is an n-ary  $\delta$-$J_1$-hyperideal of $R_1$, then $h(I_1)$ is an n-ary $\gamma$-$J_2$-hyperideal of $R_2$.
\end{theorem}
\begin{proof}
$(1)$ Let for $x_1^n \in R_1$, $g_1(x_1^n) \in h^{-1} (I_2)$. Then we get $h(g_1(x_1^n))=g_2(h(x_1^n)) \in I_2$. Since $I_2$ is a $\gamma$-$J_2$-hyperideal of $R_2$, it implies that either $h(x_i) \in J_{(m,n)}(R_2)$ which follows $x_i \in  J_{(m,n)}(R_1)$ as $h$ is a monomorphism, or 

$g_2(h(x_1),...,h(x_{i-1}),1_{R_2},h(x_{i+1}),...,h(x_n))$

$=
h(g_1(x_1^{i-1},1_{R_1},x_{i+1}^n))$

$ \in \gamma(I_2)$,\\
which follows $g_1(x_1^{i-1},1_{R_1},x_{i+1}^n) \in h^{-1}(\gamma(I_2))=\delta(h^{-1}(I_2)$. Thus $h^{-1}(I_2)$ is a $\delta$-$J_1$-hyperideal of $R_1$.

$(2)$ Let for $y_1^n \in R_2$, $g_2(y_1^n) \in h(I_1)$ such that $y_i \notin J_{(m,n)}(R_2)$. Since $h$ is an epimorphism, then there exist $x_1^n \in R_1$ such that $h(x_1)=y_1,...,h(x_n)=y_n$. Hence

$h(g_1(x_1^n))=g_2(h(x_1),...,h(x_n))=g_2(y_1^n) \in h(I_1)$.\\
Since $Ker(h) \subseteq I_1$, then we get $g_1(x_1^n) \in I_1$. Since $y_i \notin J_{(m,n)}(R_2)$, then $x_i \notin J_{(m,n)}(R_1)$. Since $I_1$ is a $\delta$-$J_1$-hyperideal of $R_1$ and $x_i \notin J_{(m,n)}(R_1)$, it implies that  $g_1(x_1^{i-1},1_{R_1},x_{i+1}^n) \in \delta(I_1)$ which implies 

$h(g_1(x_1^{i-1},1_{R_1},x_{i+1}^n))
=g_2(h(x_1),...,h(x_{i-1}),1_{R_2},h(x_{i+1}),...,h(x_n))$

$\hspace{3.4cm}=g_2(y_1^{i-1},1_{R_2},y_{i+1}^n)$

$\hspace{3.4cm}\in h(\delta(I_1))$

$\hspace{3.4cm}=\gamma(h(I_1))$\\
Thus $h(I_1)$ is a $\gamma$-$J_2$-hyperideal of $R_2$.
\end{proof}
\begin{corollary} \label{311}
Let $Q$ and $J$ be hyperideals of a Krasner $(m,n)$-hyperring $R$ such that  $I \subseteq Q$. If $Q$ is an n-ary $\delta$-$J$-hyperideal of $R$, then $Q/I$ is an n-ary  $\delta_q$ -$J$-hyperideal of  $R/I$. 
\end{corollary}
\begin{proof}
Consider the map $\pi: R \longrightarrow R/I$, defined by $r \longrightarrow f(r,I,0^{(m-2)})$. The map is a homomorphism of Krasner $(m, n)$-hyperrings, by Theorem
  3.2 in \cite{sorc1}. Now, by using Theorem \ref{45} (2), the claim can be proved. 
\end{proof}
\section{$(k,n)$-absorbing $\delta$-$J$-hyperideals }
In this section, we extend the notion of n-ary $\delta$-$J$-hyperideals to $(k,n)$-absorbing $\delta$-$J$-hyperideals.
\begin{definition}
Let $\delta$ be a hyperideal expansion of a Krasner $(m,n)$-hyperring $R$  and $k$ be a positive integer. A proper hyperideal $Q$ of $R$ is called $(k,n)$-absorbing $\delta$-$J$-hyperideal if for $x_1^{kn-k+1} \in R$, $g(x_1^{kn-k+1}) \in Q$ implies that $g(x_1^{(k-1)n-k+2}) \in J_{(m,n)}(R)$ or a $g$-product of $(k-1)n-k+2$ of $x_i^,$ s except $g(x_1^{(k-1)n-k+2})$ is in $\delta(Q)$.
\end{definition}
 \begin{example}
Suppose that $H=[0,1]$ and define 
$a \boxplus b=\{max\{a,b\}\}$, for $a \neq b$ and $a \boxplus b=[0,a]$, for $a =b$.
Let $\cdot$ is the usual multiplication on real numbers. In the Krasner $(2,3)$-hyperring $H$, the hyperideal $T=[0,0.5]$ is a $(2,2)$-absorbing $\delta_1$-$J$-hyperideal of $R$.
 \end{example}
 \begin{theorem}
 If  $Q$  is  $\delta$-$J$-hyperideal of a Krasner $(m,n)$-hyperring $R$, then $Q$ is  $(2,n)$-absorbing $\delta$-$J$-hyperideal.
 \end{theorem}
 \begin{proof}
 Let for $x_1^{2n-1} \in R$,  $g(x_1^{2n-1}) \in Q$. Since  $Q$ is a $\delta$-$J$-hyperideal of $R$, then $g(x_1^n) \in J_{(m,n)}(R)$ or $g(x_{n+1}^{2n-1}) \in \delta (Q)$. This implies that for  $1 \leq i \leq n$, $g(x_i,x_{n+1}^{2n-1}) \in \delta (Q)$, since $x_1^n \in R$ and $\delta(Q)$ is a hyperideal of $R$. Consequently, hyperideal $Q$ is $(2,n)$-absorbing $\delta$-primary.
 \end{proof}
 \begin{theorem}
 If $Q$  is  $(k,n)$-absorbing $\delta$-$J$-hyperideal of a Krasner $(m,n)$-hyperring $R$, then $Q$ is  $(s,n)$-absorbing $\delta$-$J$-hyperideal for $s>n$. 
 \end{theorem}
 \begin{proof}
 Let for $x_1^{(k+1)n-(k+1)+1} \in R$, $g(x_1^{(k+1)n-(k+1)+1}) \in Q$. Put $g(x_1^{n+2})=x$. Since hyperideal $Q$ of $R$ is  $(k,n)$-absorbing $\delta$-$J$-hyperideal, $g(x,x_{n+3}^{(k+1)n-(k+1)+1}) \in J_{(m,n)}(R)$ or a $g$-product of $kn-k+1$ of the $x_i^,$s except $g(x,x_{n+3}^{(k+1)n-(k+1)+1})$ is in $\delta(I)$. This implies that for $1 \leq i \leq n+2$, $g(x_i,x_{n+3}^{(k+1)n-(k+1)+1}) \in \delta(I)$ which means hyperideal $Q$ is  $(k+1,n)$-absorbing $\delta$-$J$-hyperideal. Thus hyperideal $Q$ of $R$ is  $(s,n)$-absorbing $\delta$-$J$-hyperideal for $s>n$.
 \end{proof}
 \begin{theorem}
 If  $Q$ is a $(k,n)$-absorbing $J$-hyperideal of a Krasner $(m,n)$-hyperring $R$, then $\sqrt{Q}^{(m,n)}$ is a $(k,n)$-absorbing  $\delta$-$J$-hyperideal.
 \end{theorem}
 \begin{proof}
 Let for $x_1^{kn-k+1} \in R$, $g(x_1^{kn-k+1}) \in \sqrt{Q}^{(m,n)}$. We suppose that none of the $g$-products of $(k-1)n-k+2$ of the $x_i^,$s other than $g(x_1^{(k-1)n-k+2})$ are  in $\delta(\sqrt{Q}^{(m,n)})$. Since $g(x_1^{kn-k+1}) \in \sqrt{Q}^{(m,n)}$ then for some $t \in \mathbb{N}$ we have for $t \leq n$, $g(g(x_1^{kn-k+1})^{(t)},1_R^{(n-t)}) \in Q$ or for $t>n$ with $t=l(n-1)+1$, $g_{(l)}(g(x_1^{kn-k+1})^{(t)}) \in Q$. In the former case, since all $g$-products  of the $x_i^,$s  other than $g(x_1^{(k-1)n-k+2})$ are not in $\delta(\sqrt{Q}^{(m,n)})$,  then they  are not in  $Q$. Since  $Q$ is a $(k,n)$-absorbing $J$-hyperideal of $R$,  then we have  $g(g(x_1^{(k-1)n-k+2)})^{(l)},1_R^{(n-t)}) \in  J_{(m,n)}(R)$ which means $g(x_1^{(k-1)n-k+2}) \in \sqrt{ J_{(m,n)}(R)}^{(m,n)}= J_{(m,n)}(R)$. By using similar argument for the second case, the claim is completed.
 \end{proof}
 \begin{theorem}
 If  $\delta(Q)$  is  $(2,n)$-absorbing $J$-hyperideal for hyperideal $Q$ of a Krasner $(m,n)$-hyperring $R$, then $Q$ is a $(3,n)$-absorbing  $\delta$-$J$-hyperideal of $R$.
 \end{theorem}
 \begin{proof}
 Let for $x_1^{3n-2} \in R$, $g(x_1^{3n-2}) \in Q$ but $g(x_1^{2n-1}) \notin  J_{(m,n)}(R)$. By $g(x_1^{3n-2}) \in Q$ it follows that $g(g(x_1,x_{2n}^{3n-2}),x_2^{2n-1}) \in Q \subseteq \delta(Q)$. Since  $\delta(Q)$  is  a $(2,n)$-absorbing $J$-hyperideal and $g(x_2^{2n-1}) \notin  J_{(m,n)}(R)$, then we have $g(x_1^n,x_{2n}^{3n-2}) \in \delta(Q)$ or $g(x_1,x_{n+1}^{2n-1},x_{2n}^{3n-2}) \in \delta(Q)$.  Hence $Q$ is a $(3,n)$-absorbing  $\delta$-$J$-hyperideal of $R$.
 \end{proof}
 \begin{theorem}
 If  $\delta(Q)$  is a $(k+1,n)$-absorbing  $\delta$-$J$-hyperideal for the hyperideal $Q$ of a Krasner $(m,n)$-hyperring $R$, then $Q$ is  $(k+1,n)$-absorbing  $\delta$-$J$-hyperideal.
 \end{theorem}
 \begin{proof}
 Let for $x_1^{(k+1)n-(k+1)+1} \in R$, $g(x_1^{(k+1)n-(k+1)+1})
 \in Q$ but $g(x_1^{kn-k+1}) \notin J_{(m,n)}(R)$.  Then we get $g(x_1^{(k+1)n-(k+1)+1})=g(x_1^{kn-k},g(x_{kn-k+1}^{(k+1)n-(k+1)+1})) \in Q \subseteq \delta(Q)$.
  Since hyperideal $\delta(I)$  is  $(k+1,n)$-absorbing $\delta$-$J$-hyperideal and  $g(x_1^{kn-k}) \notin  J_{(m,n)}(R)$, we get for $1 \leq i \leq n$, $g(x_1^{i-1},x_{i+1}^{kn-k},g(x_{kn-k+1}^{(k+1)n-(k+1)+1})) \in \delta(I)$.    Consequently, hyperideal $I$ is  a $(k+1,n)$-absorbing  $\delta$-$J$-hyperideal.
 \end{proof}
\begin{theorem}
Let $(R_1,f_1,g_1)$ and $(R_2,f_2,g_2)$ be two Krasner $(m,n)$-hyperrings and $ h:R_1 \longrightarrow R_2$ be a $\delta \gamma$-homomorphism such that $\delta$ and $\gamma$ are two hyperideal expansions of  Krasner $(m,n)$-hyperring $R_1$ and $R_2$, respectively. Then the following statements hold :

$(1)$ Let $h$ be a monomorphism. If $Q_2$ is a $(k,n)$-absorbing $\gamma$-$J_2$-hyperideal of $R_2$, then $h^{-1} (Q_2)$ is a $(k,n)$-absorbing $\delta$-$J_1$-hyperideal of $R_1$.

 $(2)$ If $h $ is an epimorphism and $Q_1$ is a $(k,n)$-absorbing $\delta$-$J_1$-hyperideal of $R_1$ such that  $Ker(h) \subseteq Q_1$,  then $h(Q_1)$ is a $\gamma$-$J_2$-hyperideal of $R_2$.
\end{theorem}
\begin{proof}
$(1)$ Let for $x_1^{kn-k+1} \in R_1$, $g_1(x_1^{kn-k+1}) \in h^{-1}(Q_2)$. It means $h(g_1(x_1^{kn-k+1}))=g_2(h(x_1),...,h(x_{kn-k+1})) \in Q_2$. Since $Q_2$ is a $(k,n)$-absorbing $\gamma$-$J_2$-hyperideal of $R_2$, we get $g_2(h(x_1),...,h(x_{(k-1)n-k+2}))=h(g_1(x_1^{(k-1)n-k+2}) \in J_{(m,n)}(R_2)$ which means $g_1(x_1^{(k-1)n-k+2}) \in J_{(m,n)}(R_1)$, as $h$ is a monomorphism, or 

$g_2(h(x_1),...,h(x_{i-1}),h(x_{i+1}),...,h(x_{kn-k+1)})=h(g_1(x_1^{i-1},x_{i+1}^{kn-k+1})) \in \gamma(Q_2)$\\
 which means  $g_1(x_1^{i-1},x_{i+1}^{kn-k+1})) \in h^{-1}(\gamma(Q_2)$ for $1 \leq i \leq n$. Since $h$ is a $\delta \gamma$-homomorphism then $g_1(x_1^{i-1},x_{i+1}^{kn-k+1})) \in \delta(h^{-1}(Q_2))$ for $1 \leq i \leq n$. Therefore we conclude that $h^{-1} (Q_2)$ is a $(k,n)$-absorbing $\delta$-$J_1$-hyperideal of $R_1$.
 
 $(2)$ Let for $y_1^{kn-k+1} \in R_2$, $g_2(y_1^{kn-k+1}) \in h(Q_1)$ such that $g_2(y_1^{(k-1)n-k+2}) \notin J_{(m,n)}(R_2)$. Then there are  $x_1^{(k-1)n-k+2} \in R_1$ such that $h(x_i)=y_i$ for $1 \leq i \leq (k-1)n-k+2$. Hence, $h(g_1(x_1^{kn-k+1})=g_2(h(x_1),...,h(x_{kn-k+1}) \in h(Q_1)$.  Since $Q_1$ containing $Ker (h)$ then  $g_1(x_1 ^{kn-k+1}) \in Q_1$. Since $Q_1$ is a $(k,n)$-absorbing $\delta$-$J_1$-hyperideal of $R_1$ and $g_1(x_1^{(k-1)n-k+2}) \notin J_{(m,n)}(R_1)$ , then  $g_1(x_1^{i-1},x_{i+1}^{kn-k+1}) \in \delta(Q_1)$ which means 
 
 $h(g_1(x_1^{i-1},x_{i+1}^{kn-k+1}))=g_2(h(x_1),...,h(x_{i-1}),h(x_{i+1}),...,h(x_{kn-k+1}))$
 
 $\hspace{3.3cm}=g_2(y_1^{i-1},y_{i+1}^{kn-k+1}) \in h(\delta(Q_1))$\\ for $1 \leq i \leq (k-1)n-k+2$.  Since $h$ is a $\delta \gamma$-epimorphism then we have $g_2(y_1^{i-1},y_{i+1}^{kn-k+1}) \in \gamma(h((Q_1))$. Consequently, $h(Q_1)$ is a $\gamma$-$J_2$-hyperideal of $R_2$.
\end{proof}

\end{document}